\documentclass{article}

\usepackage{authblk}
\usepackage{times}
\usepackage{grffile}
\usepackage{amsmath,amssymb,amsbsy}
\usepackage{amsthm}
\usepackage{hyperref}
\usepackage{color}
\usepackage{comment}
\usepackage{lipsum}
\usepackage[normalem]{ulem}
\usepackage{enumerate}
\usepackage{epstopdf}
\usepackage{bbm}
\usepackage{tikz}

\newcommand{\Rd}{\mathbb{R}^d}
\newcommand{\R}{\mathbb{R}}

\newtheorem{theorem}{Theorem}
\newtheorem{assumption}[theorem]{Assumption}

\newtheorem{lemma}[theorem]{Lemma}
\newtheorem{proposition}[theorem]{Proposition}

\title{Quantitative Flow Approximation Properties of Narrow Neural ODEs}
\author{Karthik Elamvazhuthi} 
\affil{Applied Mathematics and Plasma Physics Division\\
Los Alamos National Laboratory, USA}

\date{May 2024}

\begin{document}

\maketitle






\maketitle

\begin{abstract}%
In this note, we revisit the problem of flow  approximation properties of neural ordinary differential equations (NODEs) of the form $\dot{x} = A(t)\Sigma(W(t)x + b(t))$. The approximation properties have been considered as a flow controllability problem in recent literature. The neural ODE is considered {\it narrow} when the parameters have dimension equal to the input of the neural network, and hence have limited width. We derive the relation of narrow NODEs in approximating flows of shallow but wide NODEs of the form $\dot{x} = \sum_{i=1}^m A_i\Sigma(W_ix + b_i) $. Due to existing results on approximation properties of shallow neural networks, this facilitates understanding which kind of flows of dynamical systems can be approximated using narrow neural 
ODEs. While approximation properties of narrow NODEs have been established in literature, the proofs often involve extensive constructions or require invoking deep controllability theorems from control theory. In this paper, we provide a simpler proof technique that involves only ideas from ODEs and Gr{\"o}nwall's lemma. Moreover, we provide an estimate on the number of switches needed for the time dependent weights of the narrow NODE to mimic the behavior of a NODE with a single layer wide neural network as the velocity field. 
\end{abstract}


\section{Introduction}

The problem that we are interested in is the approximation properties of time-one map $\Psi :\R^d \rightarrow \R^d $ of the neural ODE,
\begin{eqnarray}
 \label{eq:node}
\dot{x}(t) = A(t)\Sigma(W(t)x+b(t)) ~ ~~x(0)=x_0
\end{eqnarray}
where $A: [0,T] \rightarrow \mathbb{R}^{d \times d}$, $W: [0,T] \rightarrow  \mathbb{R}^{d \times d}$ and  $b : [0,T] \rightarrow \mathbb{R}^d$ are the weights for the neural network. Here, given $\sigma :\mathbb{R} \rightarrow \mathbb{R}$, referred to as the {\it activation function}, the function $\Sigma : \mathbb{R}^d \rightarrow \mathbb{R}^d$ above is given by
\begin{equation}
\Sigma(x) = [\sigma(x_1),...,\sigma(x_d)]^T
\label{eq:vecact}
\end{equation}
The time-one map $\Psi$ is defined by setting
 $\Psi(x_0) = x(1)$ for each $x_0 \in \mathbb{R}^d$. Classical results on approximation properties of shallow neural networks do not immediately apply since $A(t)$ is fixed to be a square matrix of input dimension. For this reason, the neural ODE \eqref{eq:node} is referred to as {\it narrow} following the definition in \cite{alvarez2024interplay}. 
 
 In \cite{tabuada2022universal}, it has been shown that flows of this ODE can be used to approximate continuous maps on $\R^d$ in the uniform norm by considering a NODEs on $\mathbb{R}^{2d+1}$, provided the activation function satisfies a quadratic equation. In \cite{ruiz2023neural}, it has been shown that any $L_2$ map can be approximated using a narrow NODE. The capability NODEs for transporting probability densities have been explored in \cite{elamvazhuthi2022neural,ruiz2023neural,alvarez2024interplay}. For $L_p$ maps, universal approximations have been shown in \cite{cheng2025interpolation}.

The strategy used in this paper departs from the techniques of the other cited works, where either the arguments are constructive or control theoretic arguments are used. Firstly, we note that it can be shown that flows of  \eqref{eq:node} can approximate flows of the wide NODE
 \begin{eqnarray}
 \label{eq:node2}
\dot{x}(t) = \sum_{i=1}^mA_i\Sigma(W_ix+b_i) ~ ~~x(0)=x_0
\end{eqnarray}

This, in turn, can be used to show that flow maps of \eqref{eq:node} can be used to approximate flows of any dynamical system of the form 
 \begin{eqnarray}
 \label{eq:dynamical}
\dot{x}(t) = V(x,t) ~ ~~x(0)=x_0
\end{eqnarray}
where $V(x,t)$ is a time dependent vector field and hence, $V(x,t)$ can be approximated by an arbitrarily wide network $\sum_{i=1}^mA_i\Sigma(W_ix+b_i) $ for any $t$. Therefore, narrow neural ODEs inherit the approximation properties of their shallow but wide counterparts. This strategy is used in \cite{elamvazhuthi2022neural} and in \cite{elamvazhuthi2023learning} to study flow approximation properties using narrow NODEs. In \cite{elamvazhuthi2023learning} it was in fact shown that the dimension of the weight parameters used to approximate maps can be taken to be $m \times d$, with $m$ less than the input dimension $d$, by additionally exploiting geometric non-commutative properties of some chosen $m$ basis vector fields, owing to diffeormophism controllability results due to \cite{agrachev2009controllability}.

Our goal in this paper is to derive quantitative rates of approximation of $\eqref{eq:node}$. Existing quantitative rates are either established to approximate homoemorphism in the $L_2$ norm \cite{ruiz2023neural} or for the purposes of generative modeling \cite{ruiz2023neural,alvarez2024interplay}. In contrast, we are interested in quantitative rates for approximation of flow maps in the uniform norm. The problem we address is how many switches are needed in $A(t), W(t), b(t)$, to approximate the flow of  $\eqref{eq:node2}$. In this way, our result complements \cite{ruiz2023neural,alvarez2024interplay}. While we restrict to reference flows of the differential equations of the form \eqref{eq:node2}, an extension to general flows is straightforward since shallow but wide neural networks are known to be dense in the set of continuous functions. Hence, quantitative approximation  properties of shallow neural networks would immediately translate to quantitative approximation properties of narrow NODEs.

Before we present our analysis, we give some brief intuition behind our proof strategy. The idea behind our analysis is the following. Given a differential equation of the form,
\[\dot{x} = \sum_{i=1}^m  f_i(x):= F(x),\]
for some vector-fields $f_i(x)$, one can approximate the flow map $\Psi_F(x)$ by interatively switching between solutions of the ODEs of individual vector-fields,
\[\dot{x} =  f_i(x)\]
In our context $f_i(x) = A_i \Sigma(W_ix +b_i)$. This idea is well known from the theory of relaxed controls in control theory \cite{fattorini1999infinite}, averaging theory in dynamical systems \cite{sander2007averaging} and operator splitting in numerical methods \cite{macnamara2017operator}. In control theoretic terminology, one can think of this approximation scheme as that of using ``zeroth-order" Lie brackets to achieve approximation. That is, the non-commutativity of the vector fields is not used in any special way to approximate the flow of $F(x)$.

\section{Analysis}
For the purposes of the analysis, we will make the following assumptions on the activation function $\sigma$.
\begin{assumption}
	\label{asmp:neura}
		The activation function $\sigma$ is globally Lipschitz. That is, there exists $K>0$ such that 
		\begin{equation}
		|\sigma (x) - \sigma (y)| \leq K|x-y|, 
		\label{asmp:neura1}
		\end{equation}
		for all $x,y\in \mathbb{R}$.
\end{assumption}

	

The above assumption is satisfied by both, the Sigmoid activation function, as well as the Rectified Linear unity (ReLu) activation function.

	Let $A_i,W_i \in \mathbb{R}^{d\times d}, b_i \in \mathbb{R}^d$ be weight parameters for $i= 1,...,m$. For each $N \in \mathbb{Z}_+$. Let $Q^N$ be a $\frac{T}{N}$-periodic vector field defined by 
	\begin{equation}
	\label{eq:defosc}
	Q_{t+\frac{nT}{N}}(x) =  m A_i\Sigma(W_ix+b_i), ~~ t \in [\frac{iT}{mN},\frac{(i+1)T}{mN}), 
	\end{equation}
	for all $n \in \{0,...,N-1\}$, $i \in \{ 0,1,...,m-1\}$ and $x \in \mathbb{R}^d$. 

    Given this vector field, we consider the equation,
\begin{equation}
\label{eq:perionode}
\dot{z}(t) = Q_{t+\frac{nT}{N}}(z)
\end{equation}
Note that the vector field $Q_t$ is of the form in \eqref{eq:node} for a piecewise constant functions $A(t),W(t)$ and $b(t)$.
    
Next, we consider the differential equation defined by
   \begin{equation}
   \label{eq:avgnode}
   \dot{y}(t) = \tilde{Q}(y) = \sum_{i=1}^m A_i\Sigma(W_iy+b_i)
   \end{equation}

Before we present our resutls, we introduce some notation that will be used throughout the paper. The open ball around $x \in \Rd$ will be denoted by
$B_R(x) := \{r \in \R^d ; |x-y|  < R\}$.  For a vector $x \in \Rd$, $|x|$ will denote the $L_1$ norm of $x$. For a matrix $A \in \R^{d\times d}$, $\|A\|_1$ will denote the operator norm induced by the $L_1$ norm on $\Rd$.

We first note some preliminary bounds on the derivatives of solutions of ODEs. 
\begin{lemma}
\label{lem:derbnd}
Suppose $x_0 \in B_r(0)$ and for some $c,L>0$, $V:\Rd \times [0,T] \rightarrow \Rd$ is continuous  and uniformly Lipschitz continuous in the space variable. Additionally, assume that the vector field $V$ satisfies linear growth
\begin{equation}
|V(x,t)| \leq c+ L|x|
\end{equation}
for all $(x,t) \in \mathbb{R}^d \times [0,T]$. Then the solution $x(t)$ of the differential equation \eqref{eq:dynamical} satisfies,
\[|\dot{x} (t)| \leq c + L (r+ ct) e^{Lt} \] for almost every $t \in [0,T]$.
\end{lemma}
\begin{proof}
From the linear growth condition, we can conclude that 
\[ |x (t)| \leq |x_0| + \int_0^t c + L|x(\tau)|d\tau\]
From Gr{\"o}nwall's Lemma, this implies that 
\[ |x (t)| \leq (|x_0|+ ct) e^{Lt}\]

We know that 
\[ |V(x,t)| \leq  c+ L |x|\]
for all $x \in \Rd$ and $t \in [0,T]$.
This implies that 
\begin{align*}
|\dot{x}(t)| & \leq c + L|x(t)| \\
 &  \leq c + L (|x_0|+ ct) e^{Lt} \\
 &  \leq c + L (r+ ct) e^{Lt}
\end{align*}
\end{proof}

In the following proposition, we observe some bounds on the vector fields $\tilde{Q}(x)$ and $Q_t(x)$. We skip the proof, as one can derive it easily from the Lipschitz property of $\sigma$.

\begin{proposition}
\label{prop:bndlip}
The vector fields $\tilde{Q}$ and $Q_t$ satisfy the following linear growth bounds
\begin{equation}
|\tilde{Q}(x)|, |Q_t(x)| \leq  m   \sup_{i \in \{1,..m \}} \|A_i\|_{2} |\Sigma(b_i)| + K\|A_i\|_1  \|W_i\|_1 |x|
\end{equation}
for all $x \in \mathbb{R}^d$ and all $t \in [0,T]$, where $K$ is the Lipschitz constant of the activation function $\sigma$.

Moreover, both $Q$ and $Q_t$ are Lipschitz in $x$, uniformly with respect to time $t$, with Lipschitz constant given by
\begin{equation}
\tilde{K} = m   \sup_{i \in \{1,..N \}} K \||A_i\|_1\|W_i\|_1 
\end{equation}
\end{proposition}

Given these bounds, we can state our quantitative approximation result. The proof is heavily inspired by  \cite{artstein2007young} where the author derives quantitative approximation error of highly oscillatory dynamical systems in the context of averaging theory based on the concept of Young measures. For simpler exposition, we don't invoke averaging theory or the heavy machinery of Young measures. Instead, we explicitly prove the required bounds simplifying the technique of \cite{artstein2007young} for the purposes of this paper.

\begin{lemma}
Assume that the solution of \eqref{eq:perionode} and \eqref{eq:avgnode} satisfy $z(0), y(0) \in B_r(0)$. Then
\begin{equation}
|z(t) - y(t)| \leq \Big ( 2\frac{T}{N}  X +    \tilde{K} X \frac{ T^2}{2N} ) e^{\tilde{K}T}
\end{equation}
for all $t \in [0,T]$, where the constants $X, \tilde{K}, C >0 $ are given by
\begin{align*}
& X = c   + L(r+ct)e^{Lt}  \\
& \tilde{K} =   m   \sup_{i \in \{1,..N \}} K \|A_i\|_1\|W_i\|_1  \\
& c= m   \sup_{i \in \{1,..N \}} \|A_i\|_1|\Sigma(b_i)| \\
& L= K\|A_i\|_1  \|W_i\|_1 
\end{align*}
\end{lemma}

\begin{proof}
We compute the difference
\begin{align}
|z(t) - y(t)| & \leq  |\int^t_0 Q_t(z(\tau)) d\tau -\int^t_0 \tilde{Q}(y(\tau)) d\tau| \nonumber \\
& \leq  |\int^t_0 Q_t(z(\tau)) d\tau -\int^t_0 \tilde{Q}(z(\tau)) d\tau|+|\int^t_0  |\tilde{Q}(z(\tau))  - \tilde{Q}(y(\tau))| d\tau \nonumber
\\
& \leq  |\int^t_0 Q_t(z(\tau)) d\tau -\int^t_0 \tilde{Q}(z(\tau)) d\tau|+|\int^t_0  \tilde{K}|z(\tau)  - y(\tau)| d\tau 
\label{eq:ogbnd}
 \end{align}
where the last term follows from the Lipschtz property of $\tilde{Q}$ established in \ref{prop:bndlip}.
 We bound the first term
 \begin{align*}
|\int^t_0 Q_t(z(\tau)) d\tau -\int^t_0 \tilde{Q}(z(\tau)) d\tau|
 \end{align*}
 Let $[0,T]$ be partitioned into intervals $I_i$ of size $T/N$ such that $\cup_{i=1}^N I_i = [0,T]$. Let $s_i$ be the midpoints of $I_i$ for each $i = 1,..., m$. We define the midpoint approximation $\psi :[0,T] \rightarrow \mathbb{R}$ by
 \[\psi(\tau) = z(s_i) : = x_i ~~~ \tau \in I_i.\]
Let $R = (r+cT)e^{LT}$. We note that the solutions remain in $B_R(0)$ for all $t \in [0,T]$. From the definition above, we get the bound,
 \begin{align*}
&|\int^t_0 Q_t(z(\tau)) d\tau -\int^t_0 \tilde{Q}(z(\tau)) d\tau| \\
& \leq |\int^t_0 Q_t(\psi(\tau)) d\tau -\int^t_0 \tilde{Q}(\psi(\tau)) d\tau| + \int^t_0 |Q_t(z(\tau)) - Q_t(\psi(\tau)|d\tau  \\
& + \int^t_0 |\tilde{Q}_t(z(\tau)) - \tilde{Q}_t(\psi(\tau)|d\tau \\
& \leq |\int^t_0 Q_t(\psi(\tau)) d\tau -\int^t_0 \tilde{Q}(\psi(\tau)) d\tau| + \int^t_0 \tilde{K}|z(\tau) - \psi(\tau)|d\tau  \\
& + \int^t_0 \tilde{K} |z(\tau) - \psi(\tau)|d\tau \\
& \leq \sum_{i=1}^N |\int_{I_i \cap [0,t]} Q_t(\psi(\tau)) d\tau -\int^t_0 \tilde{Q}(\psi(\tau)) d\tau| + 2\int_{I_i \cap [0,t]}  \tilde{K}|z(\tau) - \psi(\tau)|d\tau  \\
& \leq \sum_{i=1}^N | \int_{I_i \cap [0,t]} Q_t(x_i) d\tau -\int_{I_i \cap [0,t]}  \tilde{Q}(x_i) d\tau|  +  2\sum_{i=1}^N \int_{I_i \cap [0,t]} \tilde{K}  X|\tau - s_i|d\tau 
\\
& \leq \ | \int_{I_N \cap [0,t]} Q_t(x_i) d\tau -\int_{I_N \cap [0,t]}  \tilde{Q}(x_i) d\tau|  +  2\sum_{i=1}^N \int_{I_i \cap [0,t]} \tilde{K}  X|\tau - s_i|d\tau 
\\
& = \frac{T}{N}  \sup_{x \in B_{R}(0)}|\tilde{Q}(x)| +\frac{T}{N}  \sup_{x\in B_R(0)}|Q_t(x)|+   \tilde{K} X \frac{T^2}{2N} \\
&  =  \frac{2T}{N}  X + \tilde{K} X \frac{T^2}{2N}
 \end{align*}

 In the above sequence of inequalities, the first term  $|z(\tau) - \psi(\tau)|$ was bounded by $X|\tau -s_i|$ using the Lipschitz property of $z(t)$ from the bound established in Lemma \ref{lem:derbnd}. Then the result follows by an application of Gr{\"o}nwall's lemma to the bound in \eqref{eq:ogbnd}.
\end{proof}

\bibliographystyle{plain}
\bibliography{l4dc}

\end{document}